\documentclass[a4paper,numbook,babel,final,envcountsame,11pt]{article}

\usepackage{amssymb}
\usepackage{amsthm}
\usepackage{amsmath}
\usepackage{graphicx}
\usepackage{array,hhline}

\newtheorem{lemma}{Lemma}[section]
\newtheorem{theorem}[lemma]{Theorem}
\newtheorem{proposition}[lemma]{Proposition}
\newtheorem{conjecture}[lemma]{Conjecture}
\newtheorem{corollary}[lemma]{Corollary}
\theoremstyle{definition}
\newtheorem{definition}[lemma]{Definition}
\newtheorem{remark}[lemma]{Remark}

\numberwithin{equation}{section}
\numberwithin{figure}{section}

\renewcommand{\theenumi}{\roman{enumi}}

\newcommand{\Cset}{\mathcal{C}}

\newcommand{\Lset}{\mathcal{L}}
\newcommand{\Nset}{\mathcal{N}}

\newcommand{\Xset}{\mathcal{X}}

\begin{document}

\title{\huge On a correction of a property of $GC_n$ sets}         
\author{Hakop Hakopian, Vahagn Vardanyan}        
\date{}          

\maketitle

\begin{abstract}
An $n$-poised node set $\mathcal X$ in the plane is called  $GC_n$ set if the (bivariate) fundamental polynomial of each node is a product of n linear factors. A line is called
$k$-node line if it passes through
exactly $k$-nodes of $\mathcal X.$ An $(n+1)$-node line is called maximal line. The well-known conjecture of M. Gasca and J. I. Maeztu states that every $GC_n$ set has a maximal line. Untill now the conjecture has been
proved only for the cases $n \le 5.$ We say  that a node uses a line if the line is a factor in the node's fundamental polynomial. It is a simple and well-known fact that any
maximal line $M$ is used by all $\binom{n+1}{2}$ nodes in $\mathcal
X\setminus M.$
Here we consider the main result of the paper - V.~Bayramyan, H.~Hakopian,
On a new property of n-poised and $GC_n$ sets, Adv Comput Math,  {\bf․ 43}, (2017) 607-626,
stating that any $n$-node line of $GC_n$ set
is used either by exactly $\binom{n}{2}$ nodes or by exactly
$\binom{n-1}{2}$ nodes, provided that the Gasca-Maeztu conjecture is true.

In this paper we show that this result is not correct in the case $n=3.$ Namely, we bring an example of a $GC_3$ set and a $3$-node line there which is not used at all. Fortunately, then we were able to establish that this is the only possible counterexample, i.e., the above mentioned result
is true for  all $n\ge 1, n\neq 3.$

We also characterize the exclusive case $n=3$ and present some new results on the maximal lines and the usage of $n$-node lines in $GC_n$ sets.
\end{abstract}
{\bf Key words:} Polynomial interpolation, Gasca-Maeztu conjecture,
$n$-poised set, $n$-independent set, $GC_n$ set, fundamental
polynomial, maximal line.

{\bf Mathematics Subject Classification (2010):} \\
primary: 41A05, 41A63; secondary 14H50.

\section{Introduction\label{sec:intro}}
An $n$-poised set $\mathcal X$ in the plane is a node set for which the interpolation problem with bivariate polynomials of total degree not exceeding $n$ is unisolvent.
Node sets with geometric characterization: $GC_n$ sets, introduced by Chang and Yao \cite{CY77}, form an important subclass of $n$-poised set. In a $GC_n$ set the fundamental polynomial of each node is a product of n linear factors. We say  that a node uses a line if the line is a factor in the node's fundamental polynomial. Thus in a $GC_n$ set each node uses exactly $n$ lines. A line is called $k$-node line if it passes through
exactly $k$-nodes of $\mathcal X.$ It follows from the uniqueness of fundamental polynomials that any used line passes through at least two nodes. On the other hand at most $n+1$ nodes can be collinear in $GC_n$ sets and  $(n+1)$-node line is called maximal line \cite{dB}. The well-known conjecture of M. Gasca and J. I. Maeztu states that every $GC_n$ set has a maximal line. Untill now the conjecture has been
proved only for the cases $n \le 5.$  It is a simple and well-known fact that any
maximal line $M$ in $\Xset$ is used by all $\binom{n+1}{2}$ nodes in $\mathcal
X\setminus M.$ It immediately follows from the fact that if a polynomial of total degree at most $n$ vanishes at $n+1$ points of a line then the line divides the polynomial (see Proposition \ref{prp:n+1points}).
Here we consider the main result of the paper \cite{BH} by V.~Bayramyan and H.~H.,
stating that any $n$-node line of $GC_n$ set
is used either by exactly $\binom{n}{2}$ nodes or by exactly
$\binom{n-1}{2}$ nodes, provided that the Gasca-Maeztu conjecture is true.

In this paper we show that this result is not correct in the case $n=3.$ Namely, in Subsection \ref{count} we bring an example of a $GC_3$ set and a $3$-node line there which is not used at all. Fortunately, then we were able to establish that this is the only possible counterexample, i.e., the above mentioned result
is true for  all $n\ge 1, n\neq 3.$ (see the forthcoming Theorem \ref{thm:corrected}).

We also characterize the exclusive case $n=3$ (Proposition \ref{prp:n=3}) and present some new results on the maximal lines and the usage of $n$-node lines in $GC_n$ sets.
Let  $\ell$ be an $n$-node line in a $GC_n$ set $\Xset,$ where $n \geq 4.$
Namely, we prove that if there are $n$ maximal lines passing through $n$ distinct
nodes in $\ell$ then there is at least one more maximal line in $\Xset$ (Proposition \ref{prp:nmaximals}).
We also prove that  if $\Xset$ has exactly three maximal lines then there are exactly three $n$-node lines and each is used by exactly $\binom{n}{2}$ nodes (Corollary \ref{lastcrl}).

Now let us go to the exact definitions and formulations.
Denote by $\Pi_n$ the space of bivariate polynomials of total degree
at most $n:$
\begin{equation*}
\Pi_n=\left\{\sum_{i+j\leq{n}}c_{ij}x^iy^j
\right\}.
\end{equation*}
We have that
\begin{equation} \label{N=}
N:=\dim \Pi_n=(n+2)(n+1)/2.
\end{equation}
Let $\Xset$ be a set of $s$ distinct nodes (points):
\begin{equation*}
{\mathcal X}={\mathcal X}_s=\{ (x_1, y_1), (x_2, y_2), \dots , (x_s, y_s) \} .
\end{equation*}
The Lagrange interpolation problem is: for given set of values
$$\Cset_s:=\{ c_1, c_2, \dots , c_s \}$$ find a polynomial $p \in \Pi_n$ satisfying
the conditions
\begin{equation}\label{int cond}
p(x_i, y_i) = c_i, \ \ \quad i = 1, 2, \dots s.
\end{equation}

\subsection{Poised sets}

\begin{definition}
A set of nodes ${\mathcal X}_s$ is called \emph{$n$-poised} if
for any set of values $\Cset_s$ there exists a unique polynomial
$p \in \Pi_n$ satisfying the conditions \eqref{int cond}.
\end{definition}
It is a well-known fact that if a node set $\Xset_s$ is $n$-poised then $s=N.$ Thus from now on we will consider sets $\Xset=\Xset_N$ when
$n$-poisedness is studied.

A polynomial $p \in \Pi_n$ is called an \emph{$n$-fundamental
polynomial} for a node $ A = (x_k, y_k) \in {\mathcal X}_s,$ where $1\le k\le s$,  if
\begin{equation*}
p(x_i, y_i) = \delta _{i k}, \  i = 1, \dots , s ,
\end{equation*}
where $\delta$ is the Kronecker symbol.

Let us denote the
$n$-fundamental polynomial of the node $A \in{\mathcal X}_s$ by $p_A^\star=p_{A, {\mathcal X}}^\star.$  A
polynomial vanishing at all nodes but one is also called fundamental, since it is a nonzero
constant times the fundamental polynomial.

Next proposition is a basic  Linear Algebra fact:

\begin{proposition} \label{prp:poised}
The set of nodes ${\mathcal X}_N$ is $n$-poised if and only if for each node $A\in\Xset_s$ there is
an $n$-fundamental polynomial $p_A^\star.$
\end{proposition}

\subsection{Independent sets}

\begin{definition}
A set of nodes ${\mathcal X}$ is called \emph{$n$-independent} if
each its node has an $n$-fundamental polynomial. Otherwise, ${\mathcal
X}$ is called \emph{$n$-dependent.}
\end{definition}
\noindent Clearly fundamental polynomials are linearly independent.
Therefore a necessary condition of $n$-independence is $|{\mathcal
X}| \le N.$

By using the Lagrange formula:
\begin{equation*}
p = \sum_{A\in{\mathcal X}} c_{A} p_{{A}, {\mathcal X}}^\star
\end{equation*}
for $n$-independent node set ${\mathcal X}$
we get a polynomial $p \in \Pi_n$ satisfying the
interpolation conditions \eqref{int cond}.

This implies that
a node set ${\mathcal X}={\mathcal X}_s$ is $n$-independent if and only
if the interpolation problem \eqref{int cond} is
\emph{$n$-solvable,} meaning that for any data $\mathcal C_s$ there exists a (not necessarily unique) polynomial $p \in \Pi_n$
satisfying the conditions \eqref{int cond}.

In the sequel we will use the following result on $n$-independence:

\begin{theorem}[Eisenbud, Green, Harris, \cite{EGH}]\label{thm:2n+1}
A node set consisting of at most $2n+1$-nodes is
$n$-dependent if and only if $n+2$ its nodes are collinear.
\end{theorem}
For a simple and short proof of this result see \cite{V}.
\subsection{Maximal lines}

\begin{definition}
Given an $n$-poised set $ {\mathcal X}.$ We say that a node
$A\in{\mathcal X}$ \emph{uses a line $\ell\in \Pi_1$,} if
\begin{equation*}
  p_A^\star = \ell q, \ \text{where} \ q\in\Pi_{n-1}.
\end{equation*}
\end{definition}
The following proposition is well-known (see, e.g., \cite{HJZ09b}
Proposition 1.3):
\begin{proposition}\label{prp:n+1points}
Suppose that a polynomial $p \in
\Pi_n$ vanishes at $n+1$ points of a line $\ell.$ Then we have that
\begin{equation*}
p = \ell r, \ \text{where} \ r\in\Pi_{n-1}.
\end{equation*}
\end{proposition}

\noindent Thus at most $n+1$ nodes of an $n$-poised set $\Xset$ can be collinear.
A line $M$ passing through $n+1$
nodes of the set ${\mathcal X}$ is called a \emph{maximal line} (see \cite{dB}).
Clearly, $M$ is used by all the nodes in $\Xset\setminus L.$

Next we bring some more corollaries.

\begin{corollary}\label{prp:minusmax}
Let $M$ be a maximal line of  a $GC_n$ set ${\mathcal X}.$  Then the set  ${\mathcal X}\setminus M$ is a $GC_{n-1}$ set.
Moreover, for any node $A\in \mathcal X\setminus M$ we have that
\begin{equation}\label{aaaaa}p_{A, {\mathcal X}}^\star= M p_{A, {\{\mathcal X\setminus M}\}}^\star.\end{equation}
\end{corollary}
Denote by  $ \mu(\Xset)$ the number of maximal lines of the node set  $\Xset.$
\begin{corollary}\label{properties}
Let ${\mathcal X}$ be an $n$-poised set. Then we have that
\vspace{-1.5mm} \begin{enumerate} \setlength{\itemsep}{0mm}
\item
Any two maximal lines of $\Xset$ intersect necessarily at a node of $\Xset;$
\item
Any three maximal lines of $\Xset$ cannot be concurrent;
\item   $ \mu(\Xset)\le n+2.$
\end{enumerate}
\end{corollary}

It is easily seen that $n$ lines $M_1,\ldots,M_n$ are maximal for an $n$-poised set $\Xset$ if and only if there is a node $O\in \Xset \setminus \left(M_1\cup \cdots \cup M_n\right)$ called  \emph{outside} node such that the following representation takes place

\begin{equation}\label{01O} \Xset\backslash \{O\} =\Xset_0 \cup \Xset_1,\end{equation}
$$\Xset_0 =\left\{ A_{ij} : 1\le i<j\le n\right\},
\Xset_1 =\left\{ A_{i}, A_i' : 1\le i\le n\right\},$$
where $ A_{ij}=M_i\cap M_j$ is an \emph{intersection} node, and
$ A_{i}, A_i'\in M_i$ are two \emph{additional} nodes in the line $M_i.$

In the sequel we will need the following characterization of $GC_n$ sets with exactly
$n$ maximal lines:

\begin{proposition}[Carnicer, Gasca, \cite{CG00}]\label{prp:nmax}
A set $\Xset$ is a $GC_n$ set with exactly $n$ maximal lines $M_1,\ldots,M_n,$ if and only if the representation \eqref{01O} holds with the following additional properties:
\begin{enumerate}
\item  There are $3$ lines $L_1, L_2,L_3$ concurrent at the outside node $O: \ O=L_1\cap L_2 \cap L_3$ such  that  $\Xset_1\subset L_1\cup L_2 \cup L_3.$
\item No line $L_i, i=1,2,3,$ contains $n+1$ nodes of $\Xset.$
\end{enumerate}
\end{proposition}


\section{$GC_n$ sets and the Gasca-Maeztu conjecture \label{ss:GMconj}}

Now let us consider a special type of $n$-poised sets satisfying a geometric characterization (GC) property:
\begin{definition}[Chung, Yao, \cite{CY77}]
An n-poised set ${\mathcal X}$ is called \emph{$GC_n$ set} if  the
$n$-fundamental polynomial of each node $A\in{\mathcal X}$ is a
product of $n$  linear factors.
\end{definition}
\noindent Thus, $GC_n$ sets are the sets each node of
which uses exactly $n$ lines.

Next we present the Gasca-Maeztu conjecture, briefly called GM conjecture:

\begin{conjecture}[Gasca, Maeztu, \cite{GM82}]\label{conj:GM}
Any $GC_n$ set possesses a maximal line.
\end{conjecture}

\noindent  Till now, this conjecture has been confirmed for the degrees
$n\leq 5$ (see \cite{B90}, \cite{HJZ14}).
For a generalization of the Gasca-Maeztu conjecture to maximal curves see \cite{HR}.

In the sequel we will make use of the following important result:

\begin{theorem}[Carnicer, Gasca, \cite{CG03}]\label{thm:CG}
If the Gasca-Maeztu conjecture is true for all $k\leq n$, then any $GC_n$ set possesses at least
three maximal lines.
\end{theorem}

\noindent Thus, in view of Corollary \ref{properties} (iii), the following holds  for any $GC_n$  set $\mathcal X:$
\begin{equation}\label{3,n+2} 3\le \mu(\Xset) \le n+2,\end{equation}
where for the first inequality   it is assumed that GM conjecture is true (for the degrees $k,\  6\le k\le n$).

We get also, in view of Corollary \ref{properties} (ii), that each node of $\Xset$
uses at least one maximal line.

\begin{proposition}[Carnicer, Gasca, \cite{CG03}]\label{prp:CG} Let $\Xset$ be a $GC_n$ set with $\mu(\Xset)\ge 3$ and let $M$ be a maximal line. Then we have that $$\mu(\Xset\setminus M)=\mu(\Xset)\quad \hbox{or}\quad \mu(\Xset)-1.$$
\end{proposition}

We get readily from here and Theorem \ref{thm:CG}
\begin{corollary} \label{crl:CG}Suppose that the Gasca-Maeztu conjecture is true for all $k\leq n.$ Suppose also that $\Xset$ is a $GC_n$ set with exactly three maximal lines and $M$ is a maximal line. Then the node set $\Xset\setminus M$ also possesses exactly three maximal lines.
\end{corollary}

\subsection{Two examples of $GC_n$ sets\label{ssec:SE}}

Here we will consider \emph{the Chung-Yao
} and \emph{the Carnicer-Gasca lattices}.

\subsubsection{The Chung-Yao natural lattice \label{ss:CY}}

Let a set $\Lset$ of $n+2$ lines be in general position, i.e., no two lines are parallel and no three lines are concurrent. Then the Chung-Yao set is defined as
the set $\Xset$ of all $\binom{n+2}{2}$ intersection points of these lines. Notice that the $n+2$ lines of $\Lset$ are maximal for $\Xset.$ Each fixed node
here is lying in exactly $2$ lines and does not belong to the
remaining $n$ lines. Observe that the product of the latter $n$ lines gives
the fundamental polynomial of the fixed node. Thus $\Xset$ is
$GC_n$ set.

Let us mention that any $n$-poised set with $n+2$ maximal lines forms a Chung-Yao lattice. In view of the relation \eqref{3,n+2} there are no $n$-poised sets with more maximal lines.

\subsubsection{The Carnicer-Gasca lattice  \label{ss:CG}}

Let a set $\Lset$ of $n+1$ lines be in general position. Then
the Carnicer-Gasca lattice $\Xset$ is defined as $\Xset:=\Xset'\cup\Xset'',$
where $\Xset'$ is the set of all intersection points of these $n+1$ lines, called primary nodes, and $\Xset''$ is a set of other $n+1$  non-collinear "free" nodes,
one in each of the lines.
We have that $|\Xset|=\binom{n+1}{2}+(n+1)=\binom{n+2}{2}.$ Each fixed "free" node
here is lying in exactly $1$ line. The product of the remaining $n$ lines gives
the fundamental polynomial of the fixed "free" node. Next, each fixed primary node
is lying in exactly $2$ lines. The product of the remaining $n-1$ lines and the line passing through the two "free" nodes in the $2$ lines gives
the fundamental polynomial of the fixed primary node. Thus $\Xset$ is a
$GC_n$ set. It is easily seen that $\Xset$ has exactly $n+1$ maximal lines, i.e., the lines of $\Lset.$

Let us mention that any $n$-poised set with exactly $n+1$ maximal lines forms a Carnicer-Gasca lattice (see \cite{CG00}, Proposition 2.4).

\subsection{The sets $\Nset_\ell$ and $\Xset_\ell$}

\begin{definition}[\cite{CG01}]  Given an $n$-poised set $\Xset$ and a  line $\ell.$ Then
\vspace{-1.5mm} \begin{enumerate} \setlength{\itemsep}{0mm}
\item
$\Xset_\ell$ is the subset of nodes of
$\Xset$ which use the line $\ell;$
\item
$\Nset_\ell$ is the subset of nodes of
$\Xset$ which do not use the line $\ell$ and do not lie in $\ell.$
\end{enumerate}
\end{definition}
Notice that
\begin{equation}\label{xlnl}
\Xset_\ell\cup\Nset_\ell=\Xset\setminus\ell.
\end{equation}
Suppose that $M$ is a maximal line of $\Xset$ and $\ell\neq M$ is any line. Then in view of the relation \eqref{aaaaa} we have that
\begin{equation}\label{rep}
\Xset_\ell\setminus M=(\Xset\setminus M)_\ell.
\end{equation}

The following proposition describes an important property of the set $\Nset_\ell:$
\begin{theorem}[Carnicer and Gasca, \cite{CG01}]\label{thm:Nell}
Let $\Xset$ be an $n$-poised set and  $\ell$ be a line.
Then the set $\Nset_\ell$ is $(n-1)$-dependent, provided that it is not empty.
Moreover, no node  $A\in\Nset_\ell$ possesses an $(n-1)$-fundamental polynomial for the set $\Nset_\ell.$
\end{theorem}

In the sequel we will use frequently the following $2$ lemmas from
\cite{CG03} (see also \cite{BH}).

\begin{lemma}[Carnicer, Gasca, \cite{CG03}]\label{lem:CG1}
Let $\Xset$ be an $n$-poised set and ${\ell}$ be a line with $|\ell\cap\Xset|\le n.$ Suppose also
that there is a maximal line $M_0$ such that
\begin{equation}\label{aaaa} M_0 \cap {\ell} \cap  \Xset =\emptyset.
\end{equation}
Then we have that
\begin{equation}\label{aaa}
\Xset_{\ell} = {(\Xset \setminus M_0)}_{\ell}.
\end{equation}
Moreover, if ${\ell}$ is an $n$-node line then we have that
$\Xset_{\ell} = \Xset \setminus ({\ell} \cup M_0),$ hence $\Xset_{\ell}$ is an $(n-2)$-poised set.
\end{lemma}

\begin{lemma}[Carnicer, Gasca, \cite{CG03}]\label{lem:CG2}
Let $\Xset$ be an $n$-poised set and ${\ell}$ be a line with $|\ell\cap\Xset|\le n.$ Suppose also
that there are two maximal lines $M', M''$ such that
\begin{equation}\label{bbbb}  M' \cap M''\cap {\ell} \in \Xset.\end{equation}

Then we have that
\begin{equation}\label{bbb}
\Xset_{\ell} = {(\Xset \setminus (M' \cup M''))}_{\ell}.
\end{equation}
Moreover, if ${\ell}$ is an $n$-node line then we have that
$\Xset_{\ell} = \Xset \setminus ({\ell} \cup M' \cup M''),$ hence $\Xset_\ell$ is an $(n-3)$-poised set.
\end{lemma}
Let us formulate the following useful (cf. \cite{BH}, Corollary 3.4)
\begin{corollary}\label{nor}
Let $\Xset$ be an $n$-poised set and ${\ell}$ be an $n$-node line.

\vspace{-1.5mm} \begin{enumerate} \setlength{\itemsep}{0mm}
\item
Suppose that  a maximal line $M_0$ satisfies the condition \eqref{aaaa}.
Then for any other maximal line $M,\ M\neq M_0,$ we have that $M\cap \Xset_\ell=n-1;$

\item
Suppose that   maximal lines $M'$ and $M''$ satisfy the condition \eqref{bbbb}. Then for any other maximal line $M,\ M\neq M', M\neq M'',$ we have that $M\cap \Xset_\ell=n-2.$

\end{enumerate}
\end{corollary}

\begin{proof}
The cases $n=1,2,$ are evident. Thus suppose that $n\ge 3.$
Assume that $M_0$ is a maximal line satisfying the condition \eqref{aaaa}.
Then, according to Lemma \ref{lem:CG1}, we have that
$\Xset_{\ell} = \Xset \setminus ({\ell} \cup M_0).$
Observe that $M_0$ is determined uniquely from this relation. Thus any other
maximal line $M$ intersects the line $\ell$ as well as $M_0$ at distinct nodes.
Therefore we get
$|M\cap \Xset_{\ell}| = |M\cap [\Xset\setminus ({\ell} \cup M_0)]| = (n+1)-2=n-1.$

Next assume that there are two
maximal lines $M'$ and $M''$ such that $M'\cap M''\cap {\ell} \in \Xset.$
Now, according to Lemma \ref{lem:CG2}, we have that $\Xset_{\ell} = \Xset
\setminus ({\ell} \cup M'\cup M'').$
Observe that any other maximal line $M$ has exactly $3$ nodes in the set ${\ell} \cup M'\cup M'').$ Indeed, $M$ intersects $\ell$ at a node since otherwise the conditions \eqref{aaaa} and \eqref{bbbb} hold simultaneously, which is a contradiction. Also, in view of  Corollary \ref{properties},
$M$ intersects $M'$ and $M''$ at two distinct nodes not belonging to $\ell.$ Therefore we get $|M\cap \Xset_{\ell}| =
|M\cap [\Xset\setminus ({\ell} \cup M'\cup M'')]| = (n+1)-3=n-2.$
\end{proof}

Finally, let us bring a result from \cite{BH} on $n$-node lines we are goming to use in the next section.
\begin{proposition}[Bayramyan, H., \cite{BH}]\label{prp:linennp}
Let $\Xset$ be an $n$-poised set and ${\ell}$ be a line passing through exactly $n$ nodes of $\Xset.$
Then the following hold:
\vspace{-1.5mm} \begin{enumerate} \setlength{\itemsep}{0mm}
\setlength{\itemsep}{0mm}
\item
$|\Xset_{\ell}| \leq \binom{n}{2};$
\item
If $|\Xset_{\ell}| \ge \binom{n-1}{2}+1$ then $|\Xset_{\ell}|=\binom{n}{2}.$ Moreover, $\Xset_{\ell}$ is an $(n-2)$-poised set and $\Xset_{\ell}=\Xset\setminus ({\ell}\cup M),$ where $M$ is a maximal line
such that $M \cap {\ell} \cap  \Xset =\emptyset;$

\item
If $\binom{n-1}{2}\geq |\Xset_{\ell}| \geq \binom {n-2}{2}+2$ then $|\Xset_{\ell}| = \binom{n-1}{2}.$
Moreover,  $\Xset_{\ell}$ is an $(n-3)$-poised set and $\Xset_{\ell}=\Xset\setminus ({\ell}\cup{\beta}),$ where ${\beta}\in \Pi_2$ is a conic such that $\Nset_{\ell}=({\beta}\setminus {\ell})\cap\Xset,$ and $|\Nset_{\ell}|=2n.$
Besides these $2n$ nodes the conic may contain at most one extra node, which necessarily belongs to ${\ell}.$
Furthermore, if the conic ${\beta}$ is reducible: ${\beta}={\ell}_1{\ell}_2$ then we have that $|{\ell}_i\cap(\Xset\setminus \ell)|=n,\ i=1,2.$
\end{enumerate}
\end{proposition}

\section{On $n$-node lines in $GC_n$ sets\label{s:nlgcn}}

Now we are in a position to present the corrected version of the main result of the paper \cite{BH} by V. Bayramyan and H. H.:
\begin{theorem}\label{thm:corrected}
Assume that Conjecture \ref{conj:GM} holds for all degrees up to
$n$. Let $\Xset$ be a $GC_n$ set, $n \ge 1,\ n\neq 3,$ and ${\ell}$ be an $n$-node line. Then we have
that
\begin{equation} \label{2bin} |\Xset_{\ell}| = \binom{n}{2}\quad \hbox{or} \quad \binom{n-1}{2}.
\end{equation}
Moreover, the following hold:
\vspace{-1.5mm} \begin{enumerate} \setlength{\itemsep}{0mm}
\item
$|\Xset_{\ell}| = \binom{n}{2}$ if and only if there is a maximal line $M_0$ such that $M_0 \cap {\ell} \cap  \Xset =\emptyset.$ In this case
we have that $\Xset_{\ell}=\Xset\setminus ({\ell}\cup M_0).$ Hence it is a $GC_{n-2}$ set;

\item
$|\Xset_{\ell}| = \binom{n-1}{2}$ if and only if there are two maximal lines $M', M'',$ such that $M' \cap M'' \cap {\ell} \in \Xset.$
In this case we have that $\Xset_{\ell}=\Xset\setminus ({\ell}\cup M'\cup M'').$ Hence it is a $GC_{n-3}$ set.

\end{enumerate}
\end{theorem}

In \cite{BH} this result is stated without the restriction $n\ne 3.$

In the next subsection we show that the statement of Theorem \ref{thm:corrected} is not correct in the case $n=3.$

\subsection{The counterexample \label{count}}

Consider a $GC_3$ set $\Xset^*$ with exactly three maximal lines: $M_1,M_2,M_3$ (see Fig. \ref{pic9}).
We have that in such sets nine nodes are lying in the maximal lines and one - $O$ is outside of them.  Also there are three $3$-node lines $\ell_1,\ell_2,\ell_3$ passing through the node $O$ (see Fig. \ref{pic2} and the Case 3 of the proof of Proposition \ref{prp:n=3}, below).  In the  set $\Xset^*$ we have also a fourth $3$-node line: $\ell^*$ which is not passing through the node $O.$

As we will see in the next proposition such a line cannot be used by any node in $\Xset.$ It is worth mentioning that this could also be verified directly.

\begin{figure}[ht] 
\centering
\includegraphics[scale=0.4]{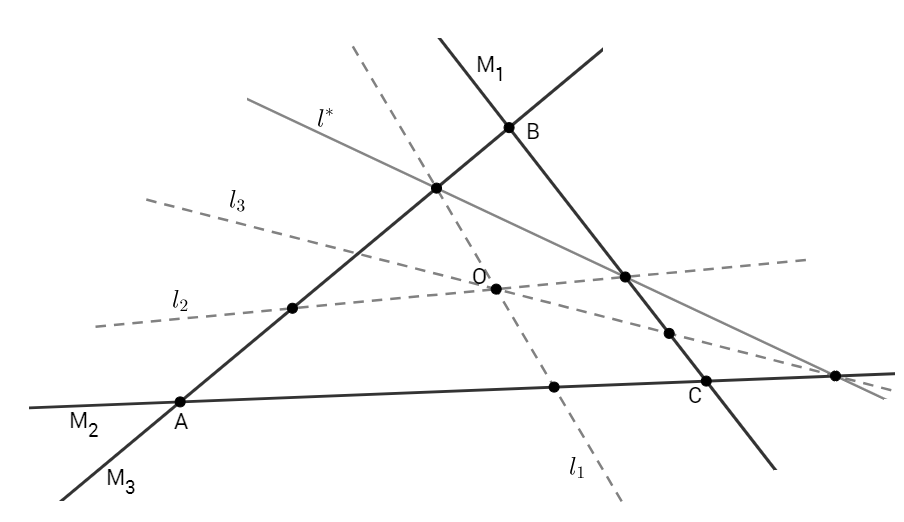}
\caption{A non-used $3$-node line in the $GS_3$ set $\Xset^*.$}\label{pic9}
\end{figure}

Before starting the proof of Theorem \ref {thm:corrected} let us
characterize the exclusive case $n=3.$

\subsection{On $3$-node lines in $GC_3$ sets\label{s:nlgcn}}

\begin{proposition}\label{prp:n=3}
Let $\Xset$ be a $GC_3$ set and ${\ell}$ be a $3$-node line. Then we have
that
\begin{equation} \label{2bin} |\Xset_{\ell}| = 3,\quad 1,\quad\hbox{or} \quad 0.
\end{equation}
Moreover, the following hold:
\vspace{-1.5mm} \begin{enumerate} \setlength{\itemsep}{0mm}
\item
$|\Xset_{\ell}| = 3$ if and only if there is a maximal line $M_0$ such that $M_0 \cap {\ell} \cap  \Xset =\emptyset.$ In this case
we have that $\Xset_{\ell}=\Xset\setminus ({\ell}\cup M_0).$ Hence it is a $GC_{1}$ set. All other maximal lines $M,\ M\neq M_0,$ intersect the line $\ell$ at a node and $M\cap\Xset_\ell$=2;

\item
$|\Xset_{\ell}| = 1$ if and only if there are two maximal lines $M', M'',$ such that $M' \cap M'' \cap {\ell} \in \Xset.$
In this case we have that $\Xset_{\ell}=\Xset\setminus ({\ell}\cup M'\cup M'');$

\item $|\Xset_{\ell}| = 0$ if and only if there are exactly three maximal lines in $\Xset$ and they intersect $\ell$ at three distinct nodes.

\end{enumerate}

Furthermore, if the node set $\Xset$ possesses exactly three maximal lines then any $3$-node line $\ell$ is used either by exactly three nodes or is not used at all:
\begin{equation} \label{1bin} |\Xset_{\ell}| = 3\quad\hbox{or} \quad 0.
\end{equation}

\end{proposition}
\begin{proof}
 First notice that the converse implications in the assertions (i) and (ii)
follow from Lemmas \ref{lem:CG1} and \ref{lem:CG2}, respectively.

The proof of the assertion (iii) and the direct implications (i) and (ii) we divide into cases depending on the number of maximal lines in $\Xset.$ Recall that according to the relation \eqref{3,n+2} the number of maximal
lines is not greater than $5$ and is not less than $3.$

Case 1. Suppose that there are exactly $5$ maximal lines in $\Xset.$ Notice that this is the case of Chung-Yao lattice.
In this case there is no $3$-node line. Indeed, through any node there pass two maximal lines in the Chung-Yao lattice.
Suppose conversely that there is a $3$-node line. Then we would have $6=3\times 2$ maximal lines passing through the $3$-nodes of $\ell$, which is a contradiction.

Case 2. Suppose that there are exactly $4$ maximal lines in $\Xset.$ Notice that this is the case of Carnicer-Gasca lattice.
The line $\ell$ is a $3$-node line. Therefore either  there is a maximal line $M$ such that $M \cap {\ell} \cap  \Xset =\emptyset$
or there are two maximal lines $M', M'',$ such that $M' \cap M'' \cap {\ell} \in \Xset.$ Thus the result holds in this case since the converse implications in the assertions (i) and (ii) are valid. Consider the case (i) when $\Xset_{\ell}=\Xset\setminus ({\ell}\cup M_0).$ Let $M,\ M\neq M_0,$ be a maximal line. Then $M$ intersects the line $\ell$ at a node, since otherwise, if $M\cap\ell \notin\Xset,$ in view of Lemma \ref{lem:CG1}, we would have that $M\cap \Xset_\ell=\emptyset$ and therefore $M\cap\Xset\subset {\ell}\cup M_0,$ which is a contradiction.
Then $M$ intersects $\ell$ and $M_0$ at two distinct nodes and the remaining two nodes of $M$ use the line $\ell,$ i.e., $|M\cap\Xset_\ell|=4-2=2.$

Case 3. Suppose that there are exactly $3$ maximal lines in $\Xset:\ M_1,M_2,M_3$  (see Fig. \ref{pic2}).

\begin{figure}[ht] 
\centering
\includegraphics[scale=0.4]{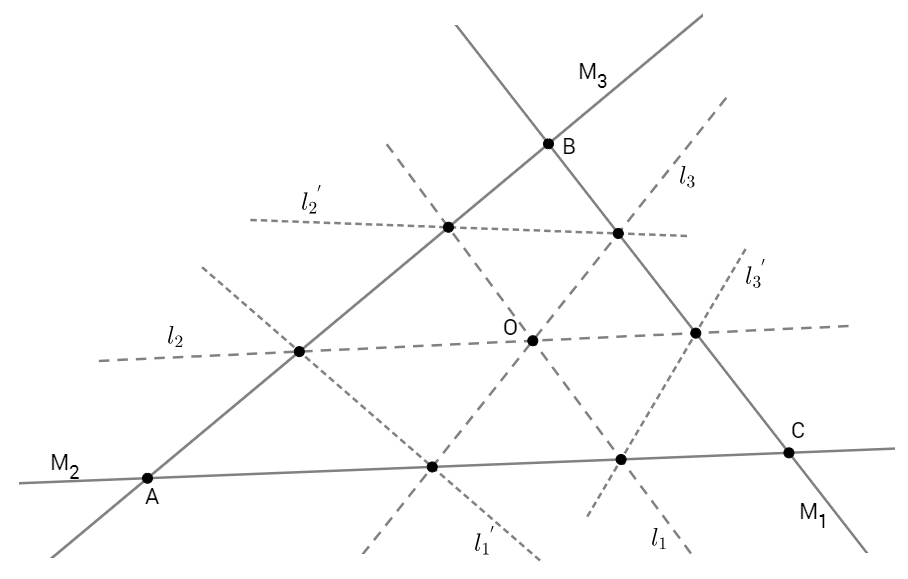}
\caption{The case of $GC_3$ set with exactly three maximal lines.}\label{pic2}
\end{figure}
By
Corollary \ref{properties} these lines form a triangle
and the vertices $A,B,C,$ are nodes in $\Xset.$  There are $6\ (=3\times 2)$
more nodes, called "free", $2$ in each maximal line. The tenth node - $O$
is outside of the maximal lines.
We find
readily that the $6$ "free" nodes are located also in $3$ lines: $\ell_1,\ell_2,\ell_3,$ passing
through $O,$ $2$ in each line (see Fig. \ref{pic2}).

To prove the result in this case it suffices to verify the following assertions:
\renewcommand{\theenumi}{\alph{enumi}}
\begin{enumerate}
\item
The lines $\ell_1,\ell_2,\ell_3,$ are $3$-node lines;

\item
For each line $\ell_1,\ell_2,\ell_3,$ there is a maximal line in $\Xset$ which does not intersect it at a node;

\item
Each of the lines $\ell_1,\ell_2,\ell_3,$ is used by exactly $3$-nodes in $\Xset;$

\item
Except of the lines $\ell_1,\ell_2,\ell_3,$ there is no other used $3$-node line in $\Xset.$

\item
Through each node of any non-used $3$-node line there pass a maximal line.

\end{enumerate}

Let us start the verification.

\begin{enumerate}
\item
Note that the lines $\ell_1,\ell_2,\ell_3,$ are $3$-node lines, i.e., they do not contain any more nodes, except $O$ and intersection nodes with two maximal lines. Indeed, otherwise they would become
a maximal line and make the number of maximal lines of $\Xset$ more than three.

\item
In view of (a) we get readily that the maximal line $M_i$ does not intersect the line $\ell_i$ at a node, $i=1,2,3.$

\item
In view of (b) and converse implication of the statement (i) we get readily that the line $\ell_i$ is used by the $3$-nodes of the set $\Xset\setminus ( \ell_i\cup M_i), \ i=1,2,3.$

\item
To verify this item let us specify all the used lines in $\Xset$ and see that except of the lines $\ell_1,\ell_2,\ell_3,$ there is no other used $3$-node line in $\Xset.$

First notice that all the nodes of the set $\Xset,$ except the vertices $A, B, C,$ use only the maximal lines and the
three $3$-node lines $\ell_1,\ell_2,\ell_3$ (see Fig. \ref{pic2}).

Then observe that each of the vertices $A, B, C,$ uses, except a maximal line and a $3$-node line $\ell_1,\ell_2,\ell_3$ also
a $2$-node line. Namely, these vertices use the lines  $\ell_1',\ell_2',\ell_3',$ respectively (see Fig. \ref{pic2}). Note that the latter lines infact are $2$-node lines. Indeed, say the line $\ell_1'$ obviously does not pass
through any more nodes from the two maximal lines that pass through the vertex $A$, since each of these maximal lines intersects $\ell_1'$ at one of its two nodes. The line $\ell_1'$ does not pass also through the remaining three nodes. Indeed, they belong to the lines $\ell_2,\ell_3$ that intersect the line $\ell_1'$ at one of its two nodes.

\item
Suppose that there is a $3$-node line different from $\ell_1,\ell_2,\ell_3$  (as $\ell^*$ in Fig. \ref{pic9}). Then clearly it is not passing through the vertices or the node $O.$ Thus the three nodes of $\ell$ belong to the maximal lines $M_i, \ i=1,2,3,$ one node to each.  Note that such a $3$-node line is not used
 since it is not among the used lines  we specified in the item (d).
\end{enumerate}

Finally, for the part ``Furthermore'', it suffices to observe that the case (ii), i.e., $|\Xset_\ell|=1,$ cannot happen if the node set $\Xset$ has exactly three maximal lines. Indeed, as it was mentioned in the item (e), it is easily seen that there cannot be a  $3$-node line passing through any intersection node of maximal lines, i.e., through $A,B,C$  (see Fig. \ref{pic2}).
\end{proof}

\section{The proof of Theorem \ref{thm:corrected}}

The original version of  Theorem \ref{thm:corrected}, i.e., the version without the restriction $n\ne 3,$ was proved in \cite{BH} by induction on $n.$ As the first step of the induction the case $n=3$ was used. We have already verified that in this case the statement is not valid. Thus to prove Theorem \ref{thm:corrected} first we need to prove it in the case $n=4.$ Note that the cases $n=1,2,$ are obvious (see \cite{BH}).

\subsection{The proof of Theorem \ref{thm:corrected} for the case $n=4$}
Consider the special case $n=4$ of Theorem \ref{thm:corrected}:
\begin{proposition}\label{prp:n=4}
Let $\Xset$ be a $GC_4$ set and ${\ell}$ be a $4$-node line. Then we have
that
\begin{equation} \label{2bin} |\Xset_{\ell}| = 6\quad \hbox{or} \quad 3.
\end{equation}
Moreover, the following hold:
\vspace{-1.5mm} \begin{enumerate} \setlength{\itemsep}{0mm}
\item
$|\Xset_{\ell}| = 6$ if and only if there is a maximal line $M_0$ such that $M_0 \cap {\ell} \cap  \Xset =\emptyset.$ In this case
we have that $\Xset_{\ell}=\Xset\setminus ({\ell}\cup M_0).$ Hence it is a $GC_{2}$ set;
\item
$|\Xset_{\ell}| = 3$ if and only if there are two maximal lines $M', M'',$ such that $M' \cap M'' \cap {\ell} \in \Xset.$
In this case we have that $\Xset_{\ell}=\Xset\setminus ({\ell}\cup M'\cup M'').$ Hence it is a $GC_{1}$ set.
\end{enumerate}
\end{proposition}

\noindent First we will prove the following

\begin{proposition}\label{prp:nmaximals}
Let $\Xset$ be a $GC_n$ set and $\ell$ be an $n$-node line, where $n \geq 1, \ n\neq 3.$ Suppose that there are $n$ maximal lines passing through $n$ distinct
nodes in $\ell.$ Then there is at least one more maximal line in $\Xset.$
\end{proposition}

\begin{proof}
Note that the cases $n=1,2$ are obvious, since then any $GC_n$ set possesses at least three maximal lines. Thus we may assume that $n\ge 4.$
  Suppose by way of contradiction that the node set $\Xset$ possesses exactly $n$ maximal lines denoted by $M_1,\ldots, M_n.$ Then the characterization of Proposition \ref{prp:nmax} with \eqref{01O} holds. Now notice that the line $\ell$ does not pass through an intersection node of two maximal lines. Indeed, in this case two maximal lines intersect $\ell$ at a node and there remain only $n-2$ maximal lines to intersect $\ell$ at other nodes. Thus clearly the condition of the Proposition cannot be satisfied. Thus the line $\ell$ may pass through only the additional nodes $A_i, A_i',\ i=1,\ldots,n,$ and the outside node $O$ (see \eqref{01O}). Thus, in view of Proposition \ref{prp:nmax}, the $n$ nodes of the line $\ell$ are lying in the three lines $L_1,L_2,L_3.$ Since $n\ge 4$ we deduce that $\ell$ coincides with a line $L_i, i=1,2,3.$ On the other hand each of these three lines intersects at most $n-1$ maximal lines at nodes of $\Xset,$ since otherwise it would become a maximal line. This contradiction proves the Proposition.
\end{proof}

\begin{remark} It is worth mentioning that Proposition \ref{prp:nmaximals} is not valid in the case $n=3.$ Indeed, the $GC_3$ set $\Xset^*$ and the $3$-node line $\ell^*$ (see Fig. \ref{pic9}) give us a counterexample for this.
\end{remark}

Now we are in a position to start

\begin{proof}[The proof of Proposition \ref{prp:n=4}.]
First notice that the converse implications in the assertions (i) and (ii)
follow from Lemmas \ref{lem:CG1} and \ref{lem:CG2}, respectively.

The proof of the direct implications we divide into cases depending on the number of maximal lines in $\Xset.$ Recall that according to the relation \eqref{3,n+2} the number of maximal
lines is not greater than $6$ and is not less than $3.$

Case 1. Assume that there are exactly $6$ maximal lines in $\Xset.$ Notice that this is the case of Chung-Yao lattice.
In this case there is no $4$-node line. Indeed, through any node there pass two maximal lines in the Chung-Yao lattice.
Suppose conversely that there is a $4$-node line. Then we would have $8=4\times 2$ maximal lines passing through the $4$ nodes of $\ell$, which is a contradiction.

Case 2. Assume that there are exactly $5$ maximal lines for $\Xset.$ Notice that this is the case of Carnicer-Gasca lattice.
The line $\ell$ is a $4$-node line. Therefore either  there is a maximal line $M$ such that $M \cap {\ell} \cap  \Xset =\emptyset$
or there are two maximal lines $M', M'',$ such that $M' \cap M'' \cap {\ell} \in \Xset.$ Therefore the result holds in this case since the converse implications in the assertions (i) and (ii) are valid.

Case 3. Assume that there are exactly $4$ maximal lines in $\Xset.$

Let us mention that if there is a maximal line $M$ such that $M \cap {\ell} \cap  \Xset =\emptyset,$ or there are two maximal lines $M', M'',$ such that $M' \cap M'' \cap {\ell} \in \Xset,$ then the statement of Theorem again follows from the converse implications in the assertions (i) and (ii).
Therefore we may suppose that the four maximal lines intersect $\ell$ in four distinct nodes.

Now, in view of the case $n=4$ of Proposition \ref{prp:nmaximals}, we conclude that there is a fifth maximal line for $\Xset,$ which contradicts our assumption.

Case 4. Assume that there are exactly $3$ maximal lines in $\Xset.$

Notice that, in view of the converse implications in the assertions (i) and (ii), we may assume that $3$ maximal lines intersect $\ell$ at $3$ distinct nodes, say first three: $A_1,A_2,A_3$ (see Fig.\ref{pic12}).

\begin{figure}[ht] 
\centering
\includegraphics[scale=0.2]{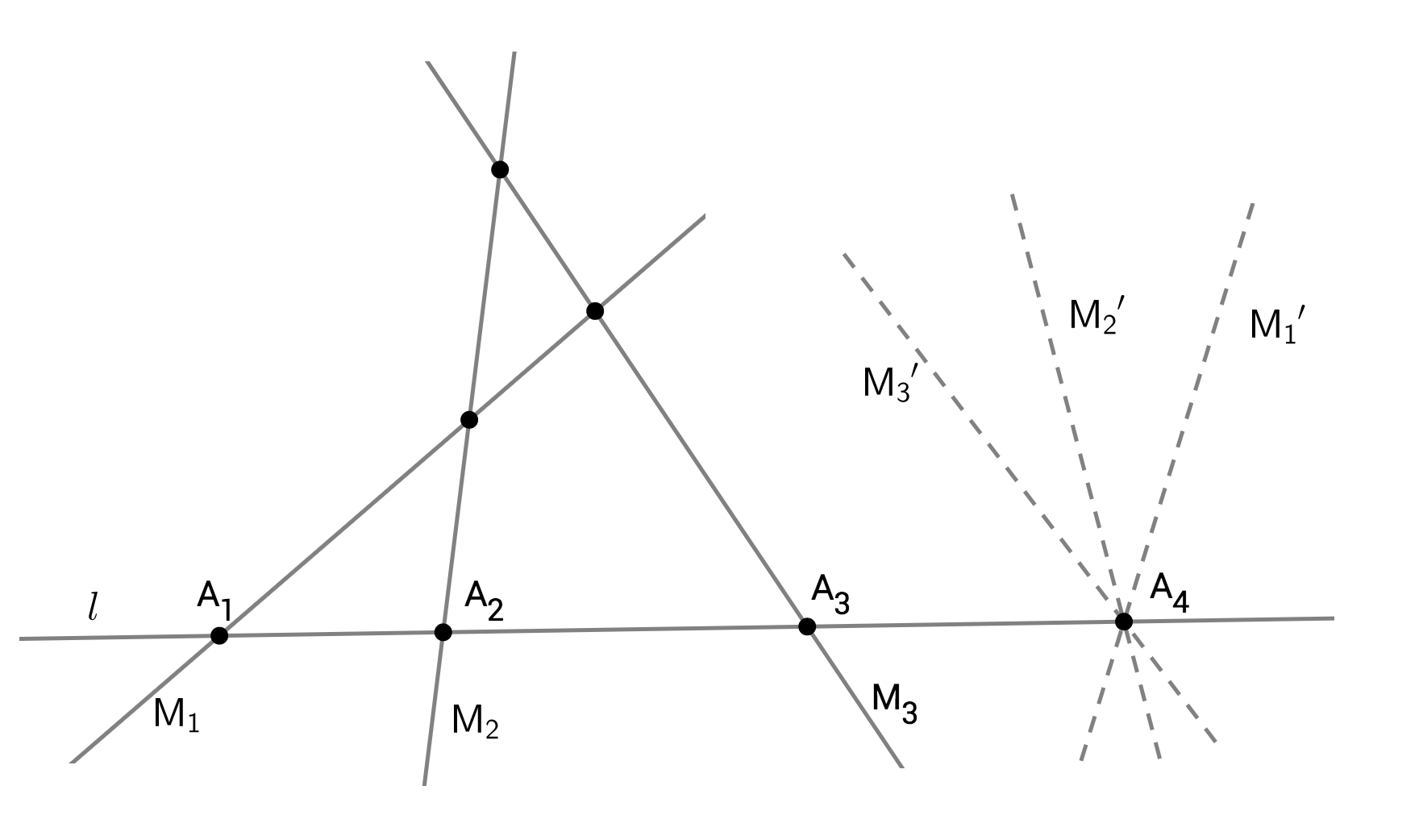}
\caption{Three maximal lines.}\label{pic12}
\end{figure}

Observe that we are to prove that this case is impossible, since the conditions in (i) or in
(ii) cannot be satisfied.

For this end it is enough to prove the following statements in this case.

(a) There is a node that uses the line $\ell,$ i.e $|\Xset_\ell|\ge 1$;

(b) If a node uses the line ${\ell}$, then there are at least $5$ nodes using it;

(c) If five nodes use the line ${\ell}$, then there is a forth maximal of the node set $\Xset.$

Obviously the statement (c) contradicts our assumption.

Now let us start with the statement (a). Assume, by way of contradiction, that $\ell$ is not used by any node of $\Xset.$ Let us consider the set $\Xset_{1} = \Xset\setminus M_{1}$. We have, in view of Corollary \ref{crl:CG}, that there are exactly $3$ maximal lines in $\Xset_1.$ Namely, $M_2, M_3,$ and a third maximal line $M_1',$  which clearly does not intersect $M_1$ at a node.  Indeed, otherwise, we would have $4$ maximal lines in the node set $\Xset.$   Now since $\ell$ is not used also in $\Xset_1$ we
obtain, in view of Proposition \ref{prp:n=3}, (iii), that in the node set $\Xset_{1}$ the third maximal line $M_{1}'$ passes through the fourth node $A_4$   of $\ell$ (see Fig.\ref{pic12}).  In a similar way the third maximal lines: $M_{2}',M_{3}'$ in the sets $\Xset_{2} = \Xset\setminus M_{2}$ and $\Xset_{3} = \Xset\setminus M_{3}$ pass through the node $A_4$ and do not intersect the maximal lines $M_2,M_3$ at nodes, respectively.

Next consider the $GC_1$ set $\Xset\setminus {(M_1\cup M_2\cup M_3)}.$  Here the lines $M_{1}',M_{2}'$ and $M_{3}'$ have each two nodes and thus are maximal, which contradicts Corollary \ref{properties} (ii).

Now let us prove the statement (b). Denote by $A$ the node that uses $\ell.$ Since the $3$ maximal lines are not concurrent we can choose a maximal line not passing through $A.$ Suppose, without loss of generality, that it is the line $M_1.$ Consider the $GC_3$ set $\Xset_1.$ As we mentioned above there are exactly $3$ maximal lines in $\Xset_1.$ Note that the node $A\in \Xset_1$ uses the line $\ell.$ On the other hand, by Proposition \ref{prp:n=3}, part "Furthermore",  the line $\ell$ can be used here either by $3$ or by no node in $\Xset_1.$ Thus we conclude that $(\Xset_1)_\ell$ consists of three noncollinear nodes. On the other hand,
we get from  Proposition \ref{prp:n=3}, (i), that there are two nodes in the maximal line
$M_2$ (as well as in $M_3$) that use the line $\ell.$ Next, there is a node in the node set $\Xset_2$ that uses the line $\ell,$ since the three nodes of $(\Xset_1)_\ell$ are noncollinear. Thus we may repeat discussion of the node set $\Xset_1$ with $\Xset_2$
and obtain that $|(\Xset_2)_\ell|=3.$ Now we have that
$$|\Xset_\ell| = |(\Xset_2)_\ell|+|(M_2\cap \Xset_\ell|\ge 3+2=5.$$
The first equality above follows from the relation \eqref{rep}.

Next let us prove the statement (c). Consider the set $\Nset_\ell.$ In view of the relation \eqref{xlnl} we get
$$|\Nset_\ell|\le 15-(4+5)=6.$$
By Theorem \ref{thm:Nell} we have that  the set $\Nset_\ell$ is $3$-dependent.
Now, Theorem \ref{thm:2n+1} implies that $5$ points from $\Nset_\ell$
are collinear, i.e., they are in a maximal line. This maximal line cannot coincide with the three maximal lines of $\Xset$, since each of them intersects $\ell$ at a node and hence has only $4$ nodes in the set $\Nset_\ell.$
Thus we get a fourth maximal line.

\end{proof}

\subsection{The proof of Theorem \ref{thm:corrected} for $n\ge 5$}

Let us mention that the proof here is similar to one from \cite{BH}, Section 3.4. But it is much shorter
due to the fact that the first step of the induction here is the case $n=4.$

Let us prove Theorem \ref{thm:corrected} by induction on $n.$
Assume that Theorem is true for all degrees less than $n$ and let us prove that it is true for the degree $n,$ where $n\ge 5.$

First assume that $|\Xset_{\ell}| \ge \binom{n-1}{2}+1.$ Then by Proposition \ref{prp:linennp} (ii), we get that
$|\Xset_{\ell}|=\binom{n}{2}$ and the direct implication in the assertion (i) holds.

Thus to prove Theorem \ref{thm:corrected} it suffices to assume that

\begin{equation} \label{a}
|\Xset_{\ell}| \le \binom{n-1}{2}
\end{equation}
and to prove that $|\Xset_{\ell}|=\binom{n-1}{2}$ and the direct implication in the assertion (ii) holds, i.e., there are two maximal lines $M', M'',$ such that $M' \cap M'' \cap {\ell} \in \Xset.$ Indeed, this will complete the proof in view of Lemma \ref{lem:CG2}.

Now let us show that there is a maximal line $M$ two nodes of which use the line ${\ell},$ i.e.,
\begin{equation} \label{bcd}
|M\cap \Xset_{\ell}|\ge 2.
\end{equation}
Indeed, we have at least three maximal lines, denoted by $M_1,M_2,M_3$ for the node set $\Xset.$ In view of Lemmas \ref{lem:CG1} and \ref{lem:CG2} we may suppose that they intersect the line $\ell$ at three distinct nodes. Now consider the $GC_{n-1}$ set $\Xset_1:=\Xset\setminus M_1.$
Here, the maximal lines $M_2,M_3,$  intersect $\ell$ at two distinct nodes. Therefore, in view of Corollary \ref{nor} and induction hypothesis, for one of them, denoted by $M,$ we have
$|M\cap (\Xset_1)_\ell| = (n-1)-1,$ or $(n-1)-2.$ Since $n\ge 5$ hence the inequality \eqref{bcd}
holds.

Now notice that, in view of \eqref{rep}, we have that
\begin{equation*}\label{repreprep}
|\Xset_{\ell}| = |(\Xset\setminus M)_{\ell}|+|M\cap \Xset_{\ell}|.
\end{equation*}
Hence, by making use of \eqref{bcd} and the induction hypothesis applied to the $GC_{n-1}$ set $\Xset\setminus M,$ we obtain that
\begin{equation}\label{fff}
 |\Xset_{\ell}| \ge |(\Xset\setminus M)_{\ell}|+2 \ge \binom{n-2}{2}+2.
\end{equation}
Therefore, in view of the condition \eqref{a} and Proposition \ref{prp:linennp} (iii), we conclude that
\begin{equation*} \label{b'}
|\Xset_{\ell}|=\binom{n-1}{2}\ \hbox{and}\ \Nset_{\ell} \subset \beta\in \Pi_2, \ |\Nset_{\ell}|=2n.
\end{equation*}
Let us use the induction hypothesis. By taking into account the first equality above and \eqref{bcd}, we deduce that
$$|(\Xset\setminus M)_{\ell}| =\binom{n-2}{2}.$$
Then we get that $2(n-1)$ nodes in $\Nset_{\ell}\cap (\Xset\setminus M)$ are located in
two maximal lines denoted by $M'$ and $M'',$ which intersect at a node $A\in {\ell}.$
Since $n\ge 5$ each of these two maximal lines passes through $4$ nodes of $\Nset_{\ell}\subset \beta.$
Thus each of them divides $\beta$ and we get $\beta=M'M''.$ Finally, according to Proposition \ref{prp:linennp} (iii), each of these lines passes through exactly $n$ nodes of $\Xset\setminus \ell.$ Therefore, since $A\in M'\cap M'',$ we get that each of these lines is maximal also for the set $\Xset.$ Hence the direct implication in the assertion (ii) holds.

At the end let us present
\begin{corollary}\label{lastcrl}
Assume that Conjecture \ref{conj:GM} holds for all degrees up to
$n$. Let $\Xset$ be a $GC_n$ set with exactly three maximal lines, where $n \ge 4.$ Then there are exactly three $n$-node lines in $\Xset$ and each of them is used by exactly $\binom{n}{2}$ nodes from $\Xset.$
\end{corollary}
\begin{proof}
Suppose that $M_1,M_2$ and $M_3$ are the three maximal lines of $\Xset.$ Let us call the intersection nodes $$A:=M_1\cap M_2,\ B:=M_2\cap M_3,\ C:=M_3\cap M_1$$
vertices. Let $\ell$ be any $n$-node line.
We are to prove that the case (ii) of Theorem \ref{thm:corrected}, i.e., $|\Xset_\ell|=\binom{n-1}{2},$ cannot happen.
For this purpose, in view of the mentioned case (ii), it suffices to show that there is no $n$-node line passing through a vertex. Assume by way of contradiction that $\ell$ is an $n$-node line passing through a vertex, say $A.$ First observe that $\ell$ intersects also $M_3$ at a node.  Indeed, otherwise in the $GC_{n-1}$ set $\Xset_3:=\Xset\setminus M_3$ the maximal lines $M_1,M_2$ and $\ell$ are concurrent at $A.$

Now let us consider the set $\Xset_{1} = \Xset\setminus M_{1}.$ We have, in view of Corollary \ref{crl:CG}, that there are exactly $3$ maximal lines in the node set $\Xset_1.$ Namely, $M_2,M_3,$ and a third maximal line denoted by $M_1'.$ Of course $M_1'$ intersects $M_2$ and $M_3$ at nodes different from vertices. Also $M_1'$ does not intersect $M_1$ at a node.  Indeed, otherwise, $M_1'$ would be the fourth maximal line in the node set $\Xset.$ 

 In a similar way the third maximal lines $M_{2}'$ and $M_{3}'$ in the sets $\Xset_{2} = \Xset\setminus M_{2}$ and $\Xset_{3} = \Xset\setminus M_{3}$ do not intersect the maximal lines $M_2$ and $M_3$ at nodes, respectively. Also $M_2'$ intersects $M_1$ and $M_3$ at nodes different from the vertices and $M_3'$ intersects $M_1$ and $M_2$ at nodes different from the vertices. From these intersection properties we readily conclude that the lines $M_1',M_2', M_3'$ are distinct $n$-node lines in $\Xset.$

We have also that the lines $M_1',M_2', M_3'$ are different from $\ell.$ Indeed, only $\ell$ from these lines passes through a vertex.

Next consider the $GC_{n-3}$ set $\Xset\setminus {(M_1\cup M_2\cup M_3)}.$  Observe that here we have four maximal lines: $M_{1}',M_{2}', M_{3}',$ and $\ell_0,$ which contradicts Corollary \ref{crl:CG}.

Finally, assume that $\ell$ is any $n$-node line in $\Xset.$  We have shown already that $|\Xset_\ell|=\binom{n}{2}.$ Therefore, in view of the case (i) of Theorem \ref{thm:corrected}, we obtain that for a maximal line $M_i,\ i=1,2,3,$ of $\Xset,$ we have that $M_i \cap {\ell} \cap  \Xset =\emptyset.$ Thus the line $\ell$ is a maximal line in the node set $\Xset_i=\Xset\setminus M_i$ and clearly it coincides with the $n$-node line $M_i'$ there. Thus  the lines $M_1',M_2', M_3',$ are the only $n$-node lines in $\Xset.$
\end{proof}

\end{document}